\documentclass[reqno, 12pt]{amsart}

\usepackage[colorlinks, linkcolor=blue, citecolor=red, urlcolor=cyan,pagebackref]{hyperref}

\allowdisplaybreaks

\usepackage{etex}
\usepackage{amsthm,amssymb,amscd,amsmath}
\usepackage{mathrsfs}
\usepackage{bbm}
\usepackage{graphicx}
\usepackage{epic,eepic}
\usepackage{tikz-cd}%for diagrams
\usepackage{stmaryrd}%for sslash
\usepackage{mathtools}
\usepackage{stmaryrd}%for llbracket

\usepackage{ascmac}%for screen environment
\usepackage[normalem]{ulem}

\usepackage{accents}

\usepackage{here}
\usepackage{bm}
\usepackage[all]{xy}

\usepackage{xcolor} 
\definecolor{mygreen}{rgb}{0,0.7,0.3}
\definecolor{myblue}{rgb}{0,0.50,1.20}%cyan={0,1.74,2.39}
\definecolor{myorange}{rgb}{1,0.5,0.1}

\definecolor{fillred}{rgb}{1,0.9,0.9}
\definecolor{fillgreen}{rgb}{0.9,1,0.9}

 \definecolor{refkey}{rgb}{0,0.7,0.3}
 \definecolor{labelkey}{rgb}{1,0,0}

\usepackage{tikz}
\usetikzlibrary{calc,decorations.markings}
\usetikzlibrary{arrows.meta}
\usetikzlibrary{positioning}
\usetikzlibrary{cd, intersections, calc, decorations.pathmorphing, arrows, decorations.pathreplacing}

\numberwithin{equation}{section}

\usepackage{version}

\usepackage{cleveref}
\crefname{thm}{Theorem}{Theorems}
\crefname{cor}{Corollary}{Corollaries}
\crefname{lem}{Lemma}{Lemmas}
\crefname{sublem}{Sublemma}{Sublemmas}
\crefname{prop}{Proposition}{Propositions}
\crefname{dfn}{Definition}{Definitions}
\crefname{defi}{Definition}{Definitions}
\crefname{ex}{Example}{Examples}
\crefname{claim}{Claim}{Claims}
\crefname{conj}{Conjecture}{Conjectures}
\crefname{conv}{Notation}{Notations}
\crefname{rem}{Remark}{Remarks}
\crefname{rmk}{Remark}{Remarks}
\crefname{prob}{Problem}{Problems}
\crefname{figure}{Figure}{Figures}
\crefname{section}{Section}{Sections}
\crefname{subsection}{Section}{Sections}
\crefname{appendix}{Appendix}{Appendices}
\crefname{introthm}{Theorem}{Theorems}
\crefname{introprop}{Proposition}{Propositions}
\crefname{introcor}{Corollary}{Corollaries}
\crefname{introconj}{Conjecture}{Conjectures}

\newtheorem{thm}{Theorem}[section]
\newtheorem{prop}[thm]{Proposition}
\newtheorem{cor}[thm]{Corollary}
\newtheorem{lem}[thm]{Lemma}

\newtheorem{introthm}{Theorem}
\newtheorem{introprop}[introthm]{Proposition}

\theoremstyle{definition}
\newtheorem{dfn}[thm]{Definition}

\theoremstyle{remark}

\newtheorem{rem}[thm]{Remark}

\newcommand{\bZ}{\mathbb{Z}}

\newcommand{\bR}{\mathbb{R}}
\newcommand{\bC}{\mathbb{C}}

\newcommand{\bp}{{\mathbf{p}}}
\newcommand{\bbp}{{\boldsymbol{p}}}
\newcommand{\br}{{\mathbf{r}}}

\newcommand{\cF}{\mathcal{F}}

\newcommand{\X}{\mathcal{X}}%--------------shortened

\newcommand{\ve}{\varepsilon}
\newcommand{\dbra}[1]{[ #1 ]}

\newcommand{\bD}{{\accentset{\mbox{\bfseries .}}{\tri} }}

\newcommand{\inv}{\mathrm{inv}}

\DeclareMathOperator{\Log}{\mathrm{Log}}

\newcommand{\tri}{\triangle}
\newcommand{\sgn}{\mathrm{sgn}}

\newcommand{\pos}{\mathbb{R}_{>0}}

\newcommand{\Teich}{Teichm\"uller}
\newcommand{\interior}{{\mathrm{int}}}

\newcommand{\op}{\mathrm{op}}

\newcommand{\opU}{\mathsf{U}}

\newcommand{\opP}{\mathsf{P}}

\newcommand{\opX}{\mathsf{X}}

\newcommand{\opA}{\mathsf{A}}
\newcommand{\opT}{\mathsf{T}}
\newcommand{\opS}{\mathsf{S}}

\newcommand{\bmu}{\boldsymbol{\mu}}

\newcommand{\ket}[1]{| #1 \rangle}

\newcommand{\mb}[1]{\underline{#1}}

\newcommand{\tdot}{\scalebox{0.4}{$\bullet$}}

\makeatletter
\newcommand{\hatt}[1]{\mathpalette\hatt@{#1}}
\newcommand{\hatt@}[2]{%
  \ooalign{%
    $\m@th#1#2$\cr
    \hidewidth
    \hfil
    \scalebox{1.5}[1]{%
      $\m@th#1\hat{\phantom{#2}}$%
    }%
    \hfil
    \hidewidth\cr
  }%
}
\makeatother

\makeatletter
\newcommand{\checkk}[1]{\mathpalette\checkk@{#1}}
\newcommand{\checkk@}[2]{%
  \ooalign{%
    $\m@th#1#2$\cr
    \hidewidth
    \hfil
    \scalebox{1.15}[1]{%
      $\m@th#1\check{\phantom{#2}}$%
    }%
    \hfil
    \hidewidth\cr
  }%
}
\makeatother

\allowdisplaybreaks

\makeatletter
\newcommand{\oset}[3][0ex]{%
  \mathrel{\mathop{#3}\limits^{
    \vbox to#1{\kern-2\ex@
    \hbox{$\scriptstyle#2$}\vss}}}}
\makeatother

\makeatletter
\newcommand{\osetnear}[3][0ex]{%
  \mathrel{\mathop{#3}\limits^{
    \vbox to#1{\kern-.3\ex@
    \hbox{$\scriptstyle#2$}\vss}}}}
\makeatother

\tikzset{pics/.cd,
handle/.style={code={
\draw (-0.72,0) to[bend left] (0.72,0);
\draw (-0.9,0.1) to[bend right] (0.9,0.1);
}}}
 
\tikzset{
  mid arrow/.style={postaction={decorate,decoration={
        markings,
        mark=at position .5 with {\arrow[#1]{stealth}}
      }}},
}

\tikzset{->-/.style 2 args={
	postaction={decorate},
	decoration={markings, mark=at position #1 with {\arrow[thick, #2]{>}}} 
    },
    ->-/.default={0.5}{}
}%refer arrow tips
\tikzset{-<-/.style 2 args={
	postaction={decorate},
	decoration={markings, mark=at position #1 with {\arrow[thick, #2]{<}}} 
    },
    -<-/.default={0.5}{}
}

\tikzset{
    partial ellipse/.style args={#1:#2:#3}{
        insert path={+ (#1:#3) arc (#1:#2:#3)}
    }
}

\pgfdeclarelayer{bg}    % declare background layer
\pgfsetlayers{bg,main}  % set the order of the layers (main is the standard layer)

\title[Phases in the cyclic quantum dilogarithm]{Phases in the cyclic quantum dilogarithm: correction to ``Cyclic quantum Teichm\"uller theory"}

\author[Tsukasa Ishibashi]{Tsukasa Ishibashi}
\address{Tsukasa Ishibashi, Mathematical Institute, Tohoku University, 
6-3 Aoba, Aramaki, Aoba-ku, Sendai, Miyagi 980-8578, Japan.}
\email{tsukasa.ishibashi.a6@tohoku.ac.jp}

\date{\today}

\begin{document}

\begin{abstract}
We first correct $\bmu_N$-phase issues in the basic definition around the cyclic quantum dilogarithm operators in \cite{Ish}. In particular, the notion of global coefficients must be extended by incorporating additional lifting data forming a $\bmu_N$-cover, and the pentagon relation generally acquires an extra $\bmu_N$-factor. Nevertheless, we prove that every point in a regular locus of the partially compactified cluster Poisson variety lifts to a \emph{strict} extended global coefficient, for which this extra $\bmu_N$-factor is trivial. Consequently, the original statements of \cite{Ish} hold for such lifts and are therefore restored.
\end{abstract}

\maketitle

\setcounter{tocdepth}{1} %depth of table of contents
\tableofcontents

\section{Introduction}\label{sec:intro}
Quantization of the \Teich\ theory can be understood in view of a (projective) representation theory of the (dotted) Ptolemy groupoid of marked surfaces. %Infinite-dimensional projective representations are constructed by Kashaev \cite{Kas98,Kas99,Kas00,Kas01} and Chekhov--Fock \cite{CF99} using the Faddeev's (non-compact) quantum dilogarithm when the quantum parameter $q=e^{i\pi \hbar}$ is generic ($\hbar >0$). 
When the quantum parameter $q=e^{i\pi \hbar}$ is a root of unity, finite-dimensional projective representations are constructed in the previous work \cite{Ish}, which in particular produces trace invariants of mapping classes. 
%Its intimate connections to earlier works of Bonahon--Liu \cite{BL}, Bai--Bonahon--Liu \cite{BBL} and Baseilhac--Benedetti \cite{BB} are also clarified. In contrast to the infinite-dimensional case \cite{Kas98,Kas99,CF99,Kas00,Kas01}, t
The basic ingredient is the cyclic quantum dilogarithm \cite{FK}. 
%See \cite{Ish} for more details.

In this short note, we give some corrections to the construction of the finite-dimensional projective representation of \cite{Ish} regarding the basic definition around the cyclic quantum dilogarithm operator. As a result, we will have an extra phase factor in the pentagon relation in general. On the other hand, we will prove that the original statements hold true for a certain choice of global coefficient, which exists over any generic point of the cluster Poisson variety. 
We remark that this problem does not affect the trace invariants, since they are already defined up to $U(1)$-phases.

\subsection{Problems to be fixed}
The problem concerns the branch-dependent
definition of the cyclic quantum dilogarithm $\Psi_\bp(X)$. Let $N \geq 3$ be an odd integer,  $q^2$ be a primitive $N$-th root of unity, and choose its square-root
$q$ so that $q^N=1$. We will use the notation
\begin{align*}
    \bmu_N:=\{q^{2k} \mid k \in \bZ_N\},
\end{align*}
which is viewed as a subgroup of $\bC^\ast$. 

Following \cite[Section 2.1]{Ish}, consider
\[
 \cF_N=\{\bp=(p^+,p^-)\in(\bC^*)^2
 \mid (p^+)^N+(p^-)^N=1 \}.
\]
For $\bp \in \cF_N$, let $y:=p^+/p^- \in \bC^\ast$.
The original definition was
\begin{equation}\label{eq:product-definition}
 \Psi_{\boldsymbol p}(X)
 :=\prod_{j=0}^{N-1}
   \bigl(p^-+q^{2j+1}p^+X\bigr)^{j/N},
\end{equation} 
where we need to choose a branch of the $N$-th root.

\subsubsection{Two minor corrections}
First, the projection $\bp\mapsto y$ is a $\bZ_N$-covering
over
\[
 B_N:=\bC^*\setminus\{y\mid y^N=-1\},
\]
not over all of $\bC^*$. Let $\pi: \cF_N \to B_N$ be the corrected $\bZ_N$-covering. The monodromy around each point $y \in \bC^\ast \setminus B_N$ generates $\bmu_N$. 
Second, the radial cuts must include their branch points.  Thus the natural
domain is
\begin{equation}\label{eq:domain}
 D_N(y)=\bC\setminus\bigcup_{a=0}^{N-1}
 \left\{X\mathrel{\Big|}|yX|\geq 1,
 \ \arg(yX)=\frac{\pi(2a+1)}{N}\right\},
\end{equation}
instead of $|yX|>1$. 

Then the function $\Psi_{\boldsymbol p}(X)$ is first defined as a convergent power series on the domain $|yX|<1$, and then analytically continued to the corrected domain $D_N(y)$. Then $\Psi_\bp:D_N(y) \to \bC$ is well-defined as a holomorphic function. This construction handles the dependence on $X$ correctly, but our point will be its $\bp$-dependence.
%The cuts in \eqref{eq:domain} are permuted by$X\mapsto q^2X$.

\subsubsection{The branch anomaly}\label{subsub:branch_anomaly}
We not that the difference equation \cite[Lemma 2.5]{Ish}
\begin{equation}\label{eq:claimed-difference}
 \Psi_{\bp}(q^2X)
 =(p^-+qp^+X)\Psi_{\bp}(X),
 \qquad X^N=1,
\end{equation}
cannot hold for every lift $\bp\in\cF_N$ if
\eqref{eq:product-definition} is interpreted as a single-valued holomorphic
function on the connected domain $D_N(y)$. Note that this difference relation is fundamental in the connection to the quantum cluster algebra.

Indeed, consider two points $\bp=(p^+,p^-)$ and
$\bp'=(q^2 p^+,q^2 p^-)$ over a common point $y \in B_N$. In particular, they share the domain $D_N(y)$ of definition. 
For any choices of branches in
\eqref{eq:product-definition}, we have
\[
 \Psi_{\bp'}(X)^N
 =q^{2\sum_{j=0}^{N-1}j}\Psi_{\bp}(X)^N
 =q^{N(N-1)}\Psi_{\bp}(X)^N
 =\Psi_{\bp}(X)^N.
\]
Consequently, $\Psi_{\bp'}/\Psi_{\bp}$ is a
holomorphic $\bmu_N$-valued function on the connected domain $D_N(y)$, and hence is
constant in $X$. Therefore
\[
 \frac{\Psi_{\bp'}(q^2X)}
      {\Psi_{\bp'}(X)}
 =
 \frac{\Psi_{\bp}(q^2X)}
      {\Psi_{\bp}(X)}.
\]
On the other hand, by applying \eqref{eq:claimed-difference}, the left-hand side produces $q^2(p^-+qp^+X)$ while the right-hand
side is $p^-+qp^+X$. This is a contradiction for generic $X$.

Equivalently, the assertion that
\[
 \prod_{j=1}^{N}
 \bigl(p^-+q^{2j+1}p^+X\bigr)^{-1/N}=1
\]
in the proof of \cite[Lemma 2.5]{Ish} is valid only on a suitably normalized sheet.  Evaluation at the boundary point 
$(p^+,p^-)=(0,1) \in \overline{\cF}_N \setminus \cF_N$ fixes the phase on the sheet attached to that point, not on
all sheets of the Fermat cover. A similar evaluation is used in the proof of the pentagon relation \cite[Section 9.1]{Ish}.

% If a section $s(y)$ is fixed and
% $\bp=\lambda(\bp)s(y)$ with
% $\lambda(\bp)\in\bmu_N$, the corresponding corrected formula is
% \[
%  \Psi_{\bp}(q^2X)
%  =\lambda(\bp)^{-1}(p^-+qp^+X)
%   \Psi_{\bp}(X).
% \]

\subsection{A solution: branch-free spectral definition}\label{subsec:strategy}
For the functional calculus, we only need the values on the spectrum $\bmu_N$. 
Recall
\begin{align*}
    w(\bp | k)
 :=\prod_{j=1}^{k}(p^-+q^{2j-1}p^+)
\end{align*}
for $k \in \bZ_N$. On a suitably normalized sheet,  \cite[Lemma 2.7]{Ish} gives the relation
\begin{align*}
    \Psi_{\bp}(q^{2k})= \Psi_{\bp}(1) w(\bp|k).
\end{align*}
In this description, the branch problem concentrates on the part ``$\Psi_{\bp}(1)$''. Instead of trying to fix this branch, we choose an independent scalar $c \in \bC^\ast$ and define
\begin{equation}\label{eq:spectral-values}
 \Psi_{\bp;c}(q^{2k})
 :=c \cdot w(\bp | k).
\end{equation}
Then for a cyclic operator $\opX \in \mathrm{End}(V_N)$ having the spectrum $\bmu_N$, the functional calculus provides a well-defined operator $\Psi_{\bp;c}(\opX)$. In order to ensure the normalization $\det \Psi_{\bp;c}(\opX)=1$, we impose the condition 
\begin{align*}
    c^N=A(\bp)^{-1},
\end{align*}
where $A(\bp):=\prod_{k \in \bZ_N} w(\bp|k)$ is a holomorphic function on $\cF_N$. Therefore, our modified cyclic quantum dilogarithm is associated to a point in the extended Fermat curve
\begin{align*}
    \widehat{\cF}_N:= \{(\bp;c) \in \cF_N \times \bC^\ast \mid c^N = A(\bp)^{-1}\}.
\end{align*}
This modified definition satisfies the $q$-difference relation (\cref{lem:difference_eq}) as originally asserted, and the inversion relation (\cref{prop:inversion}) for an extended involution on $\widehat{\cF}_N$. However, due to the unavailability of the normalization argument at the boundary point $(p^+,p^-)=(0,1) \in \overline{\cF}_N \setminus \cF_N$, it is inevitable to have an undetermined $\bmu_N$-factor in the pentagon relation (\cref{thm:pentagon}). 

We remark that the spectral definition \eqref{eq:spectral-values} has already appeared essentially in the works of Baseilhac--Benedetti (e.g. \cite[(27)]{BB-GT}) and Kim \cite[Definition 2.17]{Kim19}, where the parameters corresponding to $c$ are treated as extra scalar inputs. It is referred to as the `determinant-like' normalization in \cite[Section 8]{BB-GT}. 
Kim obtains an undetermined non-zero scalar in the pentagon relation \cite[Theorem 4.2]{Kim19}, which is restricted to $\bmu_N$ under our determinant normalization.
On the other hand, our covering space $\widehat{\cF}_N$ appears to be novel.

% \begin{align*}
%     \Psi_{\bp;c}(\opX)\ket{\underline{k}_X} := \Psi_{\bp;c}(q^{2k})\ket{\underline{k}_X}
% \end{align*}
% makes sense, where $\{\ket{\underline{k}_X}\}_{k \in \bZ_N}$ is a basis of $V_N$ such that $\opX \ket{\underline{k}_X} = q^{2k} \ket{\underline{k}_X}$.

% This definition gives the operator identities
% \begin{align}
%  \Psi_{\bp;c}(q^2 \opX)
%   =(p^-+qp^+\opX)\Psi_{\bp;c}(\opX), \quad
%  \Psi_{\bp;c}(q^{-2}\opX)
%   =(p^-+q^{-1}p^+\opX)^{-1}\Psi_{\bp;c}(\opX)
% \end{align}
% on $V_N$, as claimed in 
% In particular, all conjugation computations that use only the difference
% equations are valid for every lift $\bp$.  The factors in
% $w(\bp | k)$ are nonzero on $\cF_N$, so these operators are
% invertible.

%The condition $X^N=\id$ alone does not imply that $X$ is conjugate to the
%standard clock matrix: repeated eigenvalues are possible.  Simple spectrum
%must be added when that conclusion, or the determinant formula below, is
%used.  The functional calculus \eqref{eq:operator-calculus} itself does not
%require simple spectrum.

\subsection{Extended global coefficients}
Let $\Sigma$ be a marked surface. 
Let $\mathrm{Pt}_\Sigma$, $\mathrm{Pt}_\Sigma^{\tdot}$ be the associated Ptolemy and dotted Ptolemy groupoids, respectively. 
Recall the space $\cF_\Sigma$ of \emph{global coefficients} \cite[Definition 4.5]{Ish}, which consists of pairwise mutation-compatible collections 
\begin{align*}
    \bbp_\Sigma=\{\bp^\tri\}_{\tri \in \mathrm{Pt}_\Sigma}
\end{align*}
of coefficient tuples, where each $\bp^\tri=(\bp^\tri_\alpha)_{\alpha \in e(\tri)}$ assigns one $\bp^\tri_\alpha \in \cF_N$ for each edge $\alpha$ of the triangulation $\tri$ (where we allow $\bp^\tri_\alpha \in \overline{\cF}_N$ for boundary edges), and the mutation-compatibility conditions impose monomial relations between $\bp^\tri$ and $\bp^{\tri'}$ for each flip $\mu_\kappa: \tri \to \tri'$, which in particular ensures the required relation \eqref{eq:pentagon_coeff_rel} for the pentagon identity everywhere in the groupoid $\mathrm{Pt}_\Sigma$.

Now we introduce the space $\widehat{\cF}_\Sigma$ of \emph{extended global coefficients} by upgrading each edge datum to be $\hatt\bp_\alpha^\tri=(\bp_\alpha^\tri;c_\alpha^\tri) \in \widehat{\cF}_N$, %without imposing extra relations among $c_\alpha^\tri$. 
where an extra \emph{inversion condition} $c_\kappa^{\tri'}c_\kappa^\tri = \zeta_\inv$ is imposed. %Here, $\zeta_\inv$ is computed in \cref{sec:other_root}. 
An element of $\widehat{\cF}_\Sigma$ is written as
\begin{align*}
    \hatt\bbp_\Sigma=\{\hatt\bp^\tri\}_{\tri \in \mathrm{Pt}_\Sigma}
\end{align*}
with $\hatt\bp^\tri=(\hatt\bp^\tri_\alpha)_{\alpha \in e(\tri)}$. With this extended data, the flip operator for $\mu_\alpha: \bD_1 \to \bD_2$ is defined as
\begin{align*}%\label{eq:flip_op}
        \opT_{vw}:= \Psi_{\bp_\alpha^{\tri_1};c_\alpha^{\tri_1}}(\dbra{\opP_v^{-1}\opU_v\opP_w})\opS_{vw}.
\end{align*}
The construction of the projective representation given in \cite[Section 4.4]{Ish} proceeds without further modification. The following gives an immediate correction to \cite[Theorem 2]{Ish}:

\begin{introprop}[\cref{thm:proj_rep_groupoid}]\label{introthm:proj_rep_groupoid}
For any extended global coefficient $\hatt\bbp_\Sigma=\{\hatt\bp_\tri\}_{\tri \in \mathrm{Pt}_\Sigma} \in \widehat{\cF}_\Sigma$, we have a projective representation $V_\ast(\hatt\bbp_\Sigma): \mathrm{Pt}^{\tdot}_\Sigma \to \mathrm{Vect}_\bC^\op$ such that 
\begin{enumerate}
    \item $\opA_v^3=\mathsf{1}$. 
    \item $\opT_{uv}\opT_{uw}\opT_{vw}=\omega\opT_{vw}\opT_{uv}$ for some $\omega \in \bmu_N$. 
    \item $\opA_v \opT_{vw} \opA_w =\opA_w \opT_{wv} \opA_v$.
    \item $\opT_{vw}\opA_v \opT_{wv} = \zeta \opA_v \opA_w \opP_{(12)}$, where $\zeta \in \bmu_{6N}$ is a scalar determined by $q^2$. 
    %for $q^2=e^{2\pi i/N}$.%\footnote{We note that the computation of $\zeta$ in \cite{Ish} is specialized to the case $q^2=e^{2\pi i/N}$. For a general primitive $N$-th root of unity, this phase factor is computed in \cite{Kas00}.} 
\end{enumerate}
\end{introprop}
% Several comments are in order.
% \begin{itemize}
%     \item The relation (2) is a direct consequence of \cref{thm:pentagon}.
%     \item The relation (3) purely comes from a symmetry of $\opS_{vw}$, which is irrelevant from the cyclic quantum dilogarithm. 
%     \item The relation (4) follows from \cref{prop:inversion}.
% \end{itemize}
In \cite{Ish}, the explicit value of $\zeta_\inv$ and $\zeta$ are computed only for the choice $q^2=e^{2\pi i/N}$. In \cref{sec:other_root}, we will verify that they coincide with the values computed in \cite{Kas00} for any choice of a primitive $N$-th root of unity. We clarify that the value of $\zeta_\inv$ is canonically normalized in the sense of \cref{prop:positive_normalization}. 

The constructions of the trace invariants (\cite[Sections 6 and 8]{Ish}) of a mapping class are not affected by this correction, since the $\bmu_N$-ambiguity disappears by taking the absolute value. In particular, the invariance of the underlying global coefficient $\bbp \in \cF_\Sigma$ is sufficient to obtain these trace invariants.
%This $\bmu_N$-ambiguity is also irrelevant in the refined work \cite{GY26}.

\smallskip
We say that an extended global coefficient $\hatt \bbp_\Sigma \in \widehat{\cF}_\Sigma$ is \emph{strict} if the phase factor $\omega$ in \cref{introthm:proj_rep_groupoid} (2) is trivial along every pentagon cycle in $\mathrm{Pt}_\Sigma^{\tdot}$. In other words, the original theorem \cite[Theorem 2]{Ish} holds true for such a strict extended global coefficient. 

\begin{introthm}[\cref{thm:strict_coeff_exist}]
For any marked surface $\Sigma$, a strict extended global coefficient exists. Indeed, any point of
\begin{align*}
    \overline{\X}_\Sigma^{\mathrm{reg}}:= \{ \rho \in \overline{\X}_\Sigma \mid \text{$Y_\alpha^\tri(\rho) \neq -1$ for all $\tri \in \mathrm{Pt}_\Sigma$, $\alpha \in e(\tri)$}\}
\end{align*}
lifts to a strict extended global coefficient.
%Indeed, every quasi-Fuchsian representation $\rho: \pi_1(\Sigma) \to \PGL_2(\bC)$ lifts to a strict extended global coefficient. 
\end{introthm}
The statement \cite[Theorem 2]{Ish} is repaired in this sense. The remaining statements are unchanged for the following reasons:

\begin{itemize}
    \item \cite[Theorem 3]{Ish} is unchanged, since our modification of $\Psi_{\bp;c}$ by scalars does not affect its adjoint action, and the proof is valid with the corrected q-difference relation (\cref{lem:difference_eq}).
    \item \cite[Theorem 4]{Ish} is unchanged for strict extended global coefficients.
    \item Local and irreducible quantum traces: unchanged because their definitions take absolute values. Also note the determinant normalization \eqref{eq:det_normalized}.
    \item \cite[Theorem 5]{Ish} and comparison with reduced QH operators: operator equality may acquire a root-of-unity scalar depending on normalization, while the stated absolute-value invariant is unchanged.
Any statement requiring a literal normalized operator identity rather than a projective identity should be understood with the new extended coefficient.
\end{itemize}

\subsection*{General notation}

\begin{itemize}
    \item Henceforth fix an odd integer $N \geq 3$. Let $\bZ_N:=\bZ/N\bZ$. 
    \item Let $q^2$ be a primitive $N$-th root of unity. Define its square root by $q:=(q^2)^{(N+1)/2}$ so that $q^N=1$. 
    %Throughout the paper, we fix $q^2:=e^{2\pi i/N}$ and $q:=-e^{\pi i/N}$. For the other choices of primitive $N$-th root of unity, only the values of scalars $\zeta_\inv$ and $\zeta$ will change, and the structural results remain unchanged. See \cref{sec:other_root}.
    \item Let $\bmu_N:=\{q^{2k}\}_{k \in \bZ_N}$, which does not depend on the choice of $q^2$. We consider it as a multiplicative subgroup of $\bC^\ast$.
    \item Let $V_N:=\bigoplus_{k \in \bZ_N} \bC \ket{k}$ be a complex $N$-dimensional vector space with a cyclically labeled basis $\{\ket{k}\}_{k \in \bZ_N}$. 
    \item Let $\opU,\opP \in \mathrm{End}(V_N)$ be standard operators given by $\opU\ket{k}:=q^{2k}\ket{k}$, $\opP\ket{k}:=\ket{k+1}$ for $k \in \bZ_N$.
\end{itemize}
See \cite{Ish} for further details. 

\subsection*{Acknowledgements}
The author is supported by JSPS KAKENHI Grant Number JP24K16914.

\subsection*{AI disclosure}
%ChatGPT-5.6 Sol (OpenAI) pointed out the problems to be fixed, and also suggested the basic strategy for the correction. The final proofs were supplied  by the author, and the entire text is human-produced with AI assistance for language editing. 
%The author takes the full responsibility for the content of the paper.
ChatGPT-6 Astra (OpenAI) was used to identify issues in the published paper \cite{Ish}, suggest possible correction strategies, find appropriate references, and correct grammatical errors. 
%In particular, the initial ideas of the statement and the proof of \cref{lem:2A+B} are obtained in this use.
All mathematical arguments have been independently reconstructed and finalized by the author. The correction strategy as presented in \cref{subsec:strategy}, among various possibilities, is adopted by the author.
%None of the text or references were generated by AI: 
The manuscript was written entirely by the author, and all the references were independently verified. The author takes full responsibility for the content of the manuscript.

\section{Spectral definition of cyclic quantum dilogarithm}

\subsection{Extended Fermat curve and its monodromy}
Let 
\[
 \cF_N=\{\bp=(p^+,p^-)\in(\bC^*)^2
 \mid (p^+)^N+(p^-)^N=1 \}
\]
be the punctured Fermat curve, and consider the projection
\begin{align*}
    \pi: \cF_N \to B_N, \quad (p^+,p^-) \mapsto y:=p^+/p^-.
\end{align*}
Here, $B_N:=\bC^*\setminus\{y\mid y^N=-1\}$. Then $\pi$ is a $\bZ_N$-covering. 
% We also introduce a partial compactification
% \begin{align*}
%     \overline{\cF}_N:=\cF_N \cup \{(0,1)\},
% \end{align*}
% which covers $\overline{B}_N:=B_N \cup \{0\}$. These partial compactifications will be used only 
Recall that we obtain a local section
\begin{align}\label{eq:local_section}
    s(y):= \bigg(y\bigg( \frac{1}{1+y^N}\bigg)^{\! 1/N}, \bigg( \frac{1}{1+y^N}\bigg)^{\! 1/N}\bigg)
\end{align}
of $\pi$ by choosing a branch of the $N$-th root on a simply-connected domain. 

\begin{lem}\label{lem:monodromy_F}
The monodromy of the covering $\pi: \cF_N \to B_N$ around $b_j:=-q^{1-2j}$ for $j=1,\dots,N$ is $e^{2\pi i/N}$, which generates $\bmu_N$. In particular, it is a connected cyclic covering of degree $N$.
\end{lem}

\begin{proof}
Note that $B_N=\bC^\ast \setminus \{b_1,\dots,b_N\}$. 
Let $y_j(t):=b_j + \epsilon e^{-it}$ be a small loop around $b_j$. A local section along $y_j(t)$ is given by
\begin{align*}
    p_j^-(t):=\bigg( \frac{1}{1+y_j(t)^N}\bigg)^{\! 1/N}, \quad p_j^+(t):=y_j(t)p_j^-(t)
\end{align*}
by choosing a branch of the $N$-th root. We have $p_j^-(t)= (\epsilon N b_j^{N-1} e^{-it}+O(\epsilon^2))^{-1/N} \sim (\epsilon N b_j^{N-1})^{-1/N}e^{it/N}$, which produces the monodromy $e^{2\pi i/N}$. The deck transformation gives $(p^+,p^-) \mapsto (e^{2\pi i/N} p^+, e^{2\pi i/N} p^-)$, which acts transitively on the fiber. 
The other points have no monodromy.
\end{proof}
This monodromy is precisely the reason for the phenomenon discussed in \cref{subsub:branch_anomaly}.

For $\bp \in \cF_N$ and $k \in \bZ_{\geq 0}$, define
\begin{align*}
    w(\bp | k) =\prod_{j=1}^{k}(p^-+q^{2j-1}p^+) = (p^-)^k (-qy:q^2)_k.
\end{align*}
We have $w(\bp|k+N)=w(\bp|k)$, and hence we regard $w(\bp|-)$ as a function of $k \in \bZ_N$. We have the difference relation
\begin{align}\label{eq:w_diff_rel}
    w(\bp|k+1)=(p^- + q^{2k+1}p^+) w(\bp|k).
\end{align}
Let $A(\bp):= \prod_{k \in \bZ_N} w(\bp|k)$, which gives a holomorphic function $A: \cF_N \to \bC^\ast$. 

\begin{lem}
The function $A$ descends to $A: B_N \to \bC^\ast$. 
\end{lem}

\begin{proof}
Let $y:=\pi(\bp)=p^+/p^-$. 
By a direct computation, we obtain
\begin{align*}
    A(\bp) &= \prod_{k=1}^N w(\bp|k) \\
    &= \prod_{j=1}^N (p^- + q^{2j-1}p^+)^{N-j+1} \\
    &= ((p^-)^N+(p^+)^N)^{N+1}\prod_{j=1}^N (p^- + q^{2j-1}p^+)^{-j} \\
    &= \prod_{j=1}^N (p^- + q^{2j-1}p^+)^{-j} \\
    &= (p^-)^{-N(N+1)/2} \prod_{j=1}^N (1 + q^{2j-1}y)^{-j} \\
    &= (1+y^N)^{(N+1)/2} \prod_{j=1}^N (1 + q^{2j-1}y)^{-j},
\end{align*}
where we used \eqref{eq:local_section} in the last line. 
Evidently, this only depends on the ratio $y=p^+/p^-$.
\end{proof}
We will interchangeably write $A(\bp)$ or $A(y)$. Noting that $M:=(N+1)/2$ is an integer, we can further rewrite it as
\begin{align}\label{eq:A_function_y}
    A(y)=\prod_{j=1}^N (1 + q^{2j-1}y)^{M-j}.
\end{align}

\begin{dfn}[Spectral cyclic quantum dilogarithm]\ {}
\begin{enumerate}
    \item Define
    \begin{align*}
        \widehat{\cF}_N:=\{ \hatt\bp=(\bp;c) \in \cF_N \times \bC^\ast \mid c^N=A(\bp)^{-1}\},
    \end{align*}
    which is a $\bZ_N$-covering over $\cF_N$.
    We call $\bp \in \cF_N$ the \emph{underlying coefficient}, and $c \in \bC^\ast$ the \emph{normalization factor} of $\hatt\bp$. 
    %We also introduce $\widehat{\overline{\cF}}_N:=\widehat{\cF}_N \cup (\{0\}\times \bmu_N)$, which covers $\overline{\cF}_N$.
    \item Given $\hatt\bp=(\bp;c) \in \widehat{\cF}_N$, define
    \begin{align*}
        \Psi_{\hatt\bp}(q^{2k})=\Psi_{\bp;c}(q^{2k}):=c \cdot w(\bp|k)
    \end{align*}
    for $k \in \bZ_N$. We call $\Psi_{\hatt\bp}(q^{2k})$ the \emph{spectral cyclic quantum dilogarithm} associated with $(\bp;c) \in \widehat{\cF}_N$.
\end{enumerate}
\end{dfn}

\begin{lem}\label{lem:monodromy_hatF}
The covering $\widehat{\cF}_N \to \cF_N$ is a connected cyclic covering of degree $N$.
\end{lem}

\begin{proof}
Let us consider another covering
\begin{align*}
    \widehat{B}_N:= \{ (y;c) \in B_N \times \bC^\ast \mid c^N=A(y)^{-1}\}
\end{align*}
over $B_N$. From the expression \eqref{eq:A_function_y}, this covering has the monodromy along a clockwise loop around $b_j$ with deck transformation given by $c \mapsto q^{2(M-j)} c$, $j=1,\dots,N$. We have the pull-back diagram
\begin{equation*}
\begin{tikzcd}
    \widehat{\cF}_N \ar[r] \ar[d] & \cF_N \ar[d]\\
    \widehat{B}_N \ar[r] & B_N
\end{tikzcd}
\end{equation*}
of coverings. The deck transformations of 
$\cF_N \to B_N$ and $\widehat{B}_N \to B_N$ around $b_j$ are given by
$(p^+,p^-;c) \mapsto (e^{2\pi i/N}p^+,e^{2\pi i/N} p^-;c)$ by \cref{lem:monodromy_F} and $(p^+,p^-;c) \mapsto (p^+,p^-;q^{2(M-j)} c)$ from \eqref{eq:A_function_y}, respectively, which commute with each other. %Each generates the monodromy group $\bmu_N$. 
Therefore, the monodromy group of the composite covering $\widehat{\cF}_N \to B_N$ is precisely $\bmu_N \times \bmu_N$. Indeed, the monodromy around $b_j$ gives $\eta_j:=(e^{2\pi i/N},q^{2(M-j)})$ for $j=1,\dots,N$, and hence $\eta_M=(e^{2\pi i/N},1)$ and $\eta_{M-1}/\eta_M=(1,e^{2\pi i/N})$ generate $\bmu_N \times \bmu_N$.

Thus the monodromy group of $\widehat{\cF}_N \to \cF_N$ is $\bmu_N$, which shows the connectivity.
\end{proof}
From this, we see that a section $c=c(\bp)$ of $\widehat{\cF}_N \to \cF_N$ can only be locally defined. This explains why we need the normalization factor $c$ as an extra input.

\begin{dfn}[partial compactification]
The covering $\pi: \cF_N \to B_N$ is extended to a map $\pi: \overline{\cF}_N:=\cF_N \cup \{(0,1)\} \to B_N \cup \{0\}$.\footnote{In order to extend $\pi$ as a degree $N$ covering, we also need to add $(0,\eta)$ for $\eta \in \bmu_N$ to $\cF_N$. The space $\overline{\cF}_N$ is sufficient for our purpose of normalization arguments at $(0,1)$.}
Since $A(0,1)=1$, the extended degree $N$ covering
\begin{align*}
    \widehat{\overline{\cF}}_N:= \{ \hatt\bp=(\bp;c) \in \overline{\cF}_N \times \bC^\ast \mid c^N= A(\bp)^{-1}\} \to \overline{\cF}_N
\end{align*}
is locally trivial at $(0,1)$.
\end{dfn}
The functions $w(\bp|k)$ and $\Psi_{\hatt\bp}(q^{2k})$ are continuously extended to
\begin{align*}
    w(-|k): \overline{\cF}_N \to \bC, \quad \Psi_{-}(q^{2k}): \widehat{\overline{\cF}}_N \to \bC 
\end{align*}
by
$w(0,1|k)=1$ and $\Psi_{(0,1);c}(q^{2k})=c$ for $k \in \bZ_N$, $c \in \bmu_N$. We will use these extensions only for the normalization argument at $(0,1) \in \overline{\cF}_N \setminus \cF_N$ or its lift to $\widehat{\overline{\cF}}_N$.

\subsection{Functional calculus}
Let us call an operator $\opX \in \mathrm{End}(V_N)$ with simple spectrum $\{q^{2k}\}_{k\in \bZ_N}$ a \emph{cyclic operator} on $V_N$. Choose a basis $\ket{\mb{k}_\opX}$ such that $\opX\ket{\mb{k}_\opX} = q^{2k} \ket{\mb{k}_\opX}$ and define $\Psi_{\bp;c}(\opX) \in \mathrm{End}(V_N)$ by
\begin{align*}
    \Psi_{\bp;c}(\opX)\ket{\mb{k}_\opX} := \Psi_{\bp;c}(q^{2k}) \ket{\mb{k}_\opX}.
\end{align*}
The spectral cyclic quantum dilogarithm is designed so that
\begin{align}\label{eq:det_normalized}
    \det \Psi_{\bp;c}(\opX) = \prod_{k \in \bZ_N} \Psi_{\bp;c}(q^{2k}) = c^N A(\bp)= 1.
\end{align}
Compare with \cite[Lemma 2.8]{Ish}.

\subsection{$q$-difference relations}
\begin{lem}[cf. {\cite[Lemma 2.5]{Ish}}]\label{lem:difference_eq}
For any $(\bp;c) \in \widehat{\overline{\cF}}_N$ and a cyclic operator $\opX \in \mathrm{End}(V_N)$, we have the $q$-difference relations
\begin{align*}
    \Psi_{\bp;c}(q^2\opX) &= (p^- + q p^+ \opX) \Psi_{\bp;c}(\opX), \\
    \Psi_{\bp;c}(q^{-2}\opX) &= (p^- + q^{-1}p^+ \opX)^{-1} \Psi_{\bp;c}(\opX),
\end{align*}
on $V_N$.
\end{lem}

\begin{proof}
Follows from the definitions and \eqref{eq:w_diff_rel}.
\end{proof}

\subsection{Inversion relation}
Recall the Gauss coefficient $\gamma(k):=q^{k^2}$ for $k \in \bZ_N$. We have the inversion relation \cite[Theorem 2.9 (1)]{Ish}
\begin{align}\label{eq:inversion_w}
    w(\bp|k) w(\bp^{-1}|-k) = \gamma(k)
\end{align}
for $k \in \bZ_N$, where $\bp^{-1}:=(p^-,p^+)$ for $\bp=(p^+,p^-)$. 

\begin{lem}\label{lem:involution}
We have a well-defined involution $(\bp;c) \mapsto (\bp;c)^{-1}$ on $\widehat{\cF}_N$ given by 
\begin{align}\label{eq:involution}
    (\bp;c)^{-1} := (\bp^{-1}; \zeta_\inv c^{-1}),
\end{align}
where %$\zeta_\inv:=e^{\frac{\pi}{6}(N-\frac{1}{N})}$.
$\zeta_\inv:=e^{-\pi i A_{L,N}}$ (\cref{prop:positive_normalization}).
\end{lem}
We postpone explicit computations around the phase factor $\zeta_\inv$ until \cref{sec:other_root}.

\begin{proof}
We need to check that the defining relation $c^N=A(\bp)^{-1}$ is preserved. We first show that $A(\bp) A(\bp^{-1}) = \zeta_\inv^{-N}$.
%(cf. \cite[Theorem 2.9 (2)]{Ish}).
Indeed,
\begin{align*}
    A(\bp) A(\bp^{-1}) = \prod_{k \in \bZ_N} w(\bp|k)w(\bp^{-1}|-k) = \prod_{k \in \bZ_N} \gamma(k) = e^{\pi i N A_{L,N}}= \zeta_\inv^{-N}.
\end{align*}
Here we used \cref{prop:A_gamma}.
Therefore,
\begin{align*}
    A(\bp^{-1})^{-1} = \zeta_\inv^N A(\bp) = (\zeta_\inv c^{-1})^N
\end{align*}
as desired.
\end{proof}

\begin{rem}\label{rem:zeta_inv_choice}
The proof also shows that the involution \eqref{eq:involution} is canonical up to multiplication by an element of $\bmu_N$.
Indeed, if we modify the involution by $(\bp;c)^{-1} := (\bp^{-1}; \eta c^{-1})$ for some $\eta \in \bC^\ast$, then it must hold that $\eta^N=\zeta_\inv^N$. We will see in \cref{prop:positive_normalization} that the choice $\eta=\zeta_\inv$ is correctly normalized.
\end{rem}

\begin{prop}[cf. {\cite[Theorem 2.9]{Ish}}]\label{prop:inversion}
For any $(\bp;c) \in \widehat{\cF}_N$ and $k \in \bZ_N$, we have the following:
\begin{align}\label{eq:inversion}
    \Psi_{\bp;c}(q^{2k}) \Psi_{(\bp;c)^{-1}}(q^{-2k}) = \gamma(k) \zeta_\inv. 
\end{align}
\end{prop}

\begin{proof}
Follows from the definitions and \eqref{eq:inversion_w}.
\end{proof}

\subsection{Pentagon relation}
As noted in \cref{sec:intro}, the phase ambiguity in the operator pentagon relation is inevitable, and indeed depends on the choice of coefficients. Therefore, we restate it in the following way.

\begin{thm}[cf. {\cite[Theorems 3.2 and 3.3]{Ish}}]\label{thm:pentagon}
Choose $\bp,\br,\bp',\br',\bp'' \in \overline{\cF}_N$ so that
\begin{align}\label{eq:pentagon_coeff_rel}
    \begin{aligned}
    \frac{{r'}^+}{{r'}^-} &= \frac{{r}^+}{{r}^-} (p^-)^{-1}, &
    \frac{{p''}^-}{{p''}^+} &= \frac{{p}^-}{{p}^+} ({r'}^-)^{-1}, &  
    \frac{{p'}^-}{{p'}^+} &= \frac{{r'}^-}{{r'}^+} ({p''}^+)^{-1}, \\
    \frac{{p''}^+}{{p''}^-} &= \frac{{r}^-}{{r}^+} {p'}^+, &
    \frac{{p'}^+}{{p'}^-} &= \frac{{p}^+}{{p}^-} {r}^+.
    \end{aligned}
\end{align}
Choose any lifts $\hatt\bp,\hatt\br,\hatt\bp',\hatt\br',\hatt\bp'' \in \widehat{\overline{\cF}}_N$. 
Then we have the pentagon relation
\begin{align}\label{eq:pentagon}
    \Psi_{\hatt\bp}(\opU) \Psi_{\hatt\br'}(\opP) = \omega \Psi_{\hatt\br} (\opP) \Psi_{\hatt\bp'}(q^{-1}\opU\opP) \Psi_{\hatt\bp''}(\opU) 
\end{align}
on $V_N$, where $\omega=\omega(\hatt\bp,\hatt\br,\hatt\bp',\hatt\br',\hatt\bp'') \in \bmu_N$.
\end{thm}

\begin{proof}
The computation of the adjoint action in \cite[Section 9.1]{Ish} shows that 
\begin{align*}
    \Psi_{\hatt\bp}(\opU) \Psi_{\hatt\br'}(\opP) = \omega \Psi_{\hatt\br} (\opP) \Psi_{\hatt\bp'}(q^{-1}\opU\opP) \Psi_{\hatt\bp''}(\opU)
\end{align*}
holds for some scalar $\omega \in \bC^\ast$ by the irreducibility of the representation $V_N$. By taking the determinants of both sides and using $\det\Psi_{\hatt\bp}(\opX)=1$ for any cyclic operator, we obtain $\omega^N=1$.
\end{proof}

\begin{rem}
Of cource, the phase factor $\omega$ can be canceled by appropriately choosing any one of the lifts $\hatt\bp,\hatt\br,\hatt\bp',\hatt\br',\hatt\bp'' \in \widehat{\overline{\cF}}_N$, while keeping the others fixed. However, extending this `ad hoc' correction globally over the Ptolemy groupoid would encounter a cocycle problem.
\end{rem}

\section{Projective representation of the dotted Ptolemy groupoid}
In the previous section, we completed the necessary corrections concerning the basics of the cyclic quantum dilogarithm. Except for the necessity of lifts of coefficients to $\widehat{\cF}_N$, only the change in the resulting relations is the $\bmu_N$-factor in the pentagon relation \eqref{eq:pentagon}.  

Let $\Sigma$ be a marked surface. See \cite[Section 4.1]{Ish} for our convention. 
Let $\mathrm{Pt}_\Sigma$, $\mathrm{Pt}_\Sigma^{\tdot}$ be the associated Ptolemy and dotted Ptolemy groupoids, respectively. 
Recall that an object $\bD \in \mathrm{Pt}_\Sigma^{\tdot}$ is a dotted ideal triangulation of $\Sigma$, equipped with a numbering of its triangles. 

Recall the space $\cF_\Sigma$ of \emph{global coefficients} \cite[Definition 4.5]{Ish}, which consists of pairwise mutation-compatible collections 
\begin{align*}
    \bbp_\Sigma=\{\bp^\tri\}_{\tri \in \mathrm{Pt}_\Sigma}
\end{align*}
of coefficient tuples, where each $\bp^\tri=(\bp^\tri_\alpha)_{\alpha \in e(\tri)}$ assigns one $\bp^\tri_\alpha \in \cF_N$ for each edge $\alpha$ of the triangulation $\tri$ (where we allow $\bp^\tri_\alpha \in \overline{\cF}_N$ for boundary edges), and the mutation-compatibility conditions impose monomial relations between $\bp^\tri$ and $\bp^{\tri'}$ for each flip $\mu_\kappa: \tri \to \tri'$, which in particular ensures the relation \eqref{eq:pentagon_coeff_rel} in every pentagon in the groupoid $\mathrm{Pt}_\Sigma$.

\begin{dfn}
The space $\widehat{\cF}_\Sigma$ of \emph{extended global coefficients} consists of elements
\begin{align*}
    \hatt\bbp_\Sigma=\{\hatt\bp^\tri\}_{\tri \in \mathrm{Pt}_\Sigma},
\end{align*}
where $\hatt\bp^\tri=(\hatt\bp^\tri_\alpha)_{\alpha \in e(\tri)}$, $\hatt\bp^\tri_\alpha=(\bp^\tri_\alpha;c^\tri_\alpha) \in \widehat{\cF}_N$ for an interior edge $\alpha$ and $\hatt\bp^\tri_\alpha\in \widehat{\overline{\cF}}_N$ for a boundary edge $\alpha$, such that the underlying data $\bp^\tri=(\bp^\tri_\alpha)_{\alpha \in e(\tri)}$ define a global coefficient, and the \emph{inversion condition} 
\begin{align}\label{eq:inversion_condition}
    c_\kappa^{\tri'}c_\kappa^\tri = \zeta_\inv
\end{align}
holds for any flip $\mu_\kappa: \tri \to \tri'$. Under the first condition, the inversion condition is equivalent to $\hatt\bp_\kappa^{\tri'}=(\hatt\bp_{\kappa}^\tri)^{-1}$ in the notation \eqref{eq:involution}.
\end{dfn}
With this data, 
the flip operator for $\mu_\alpha: \bD_1 \to \bD_2$ is defined as
\begin{align*}%\label{eq:flip_op}
        \opT_{vw}:= \Psi_{\hatt\bp_{\alpha}^{\tri_1}}(\dbra{\opP_v^{-1}\opU_v\opP_w})\opS_{vw}: V_{\bD_2} \to V_{\bD_1},
\end{align*}
where $\Psi_{\hatt\bp_{\alpha}^{\tri_1}}$ is the spectral cyclic quantum dilogarithm associated with $\hatt\bp_{\alpha}^{\tri_1} \in \widehat{\cF}_N$. %Here, the state space $V_{\bD}(\hatt\bbp_\Sigma)=V_\bD$ is the same as the original construction, while we indicate the coefficient data for later discussion. 

The construction of the projective representation given in \cite[Section 4.4]{Ish} proceeds without further modification. Consequently, \cite[Theorem 2]{Ish} is corrected as follows:

\begin{prop}[cf. {\cite[Theorem 2]{Ish}}]\label{thm:proj_rep_groupoid}
For any extended global coefficient $\hatt\bbp_\Sigma=\{\hatt\bp_\tri\}_{\tri \in \mathrm{Pt}_\Sigma} \in \widehat{\cF}_\Sigma$, we have a projective representation $V_\ast(\hatt\bbp_\Sigma): \mathrm{Pt}^{\tdot}_\Sigma \to \mathrm{Vect}_\bC^\op$ such that 
\begin{enumerate}
    \item $\opA_v^3=\mathsf{1}$. 
    \item $\opT_{uv}\opT_{uw}\opT_{vw}=\omega\opT_{vw}\opT_{uv}$ for some $\omega \in \bmu_N$. 
    \item $\opA_v \opT_{vw} \opA_w =\opA_w \opT_{wv} \opA_v$.
    \item $\opT_{vw}\opA_v \opT_{wv} = \zeta \opA_v \opA_w \opP_{(12)}$, where %$\zeta =e^{\frac{\pi i}{6}(1-\frac 1N)} \in \bmu_{6N}$. 
    $\zeta=e^{\frac{\pi i}{3}(B_{L,N}-A_{L,N})} \in \bmu_{6N}$.
\end{enumerate}
\end{prop}
For the explicit computation of $\zeta$, see \cref{sec:other_root}. 

\begin{proof}
The relation (1) is proved in \cite{Ish}. The relation (2) is a direct consequence of \cref{thm:pentagon}. The relation (3) comes solely from a symmetry of $\opS_{vw}$, which is independent of the cyclic quantum dilogarithm. 

The essential part of the relation (4) is the inversion relation (\cref{prop:inversion}), 
which is ensured by the inversion condition \eqref{eq:inversion_condition}. Then the calculation in \cite[Section 9.2.4]{Ish} is recovered, which shows $\zeta=\zeta_\inv(\zeta_A\zeta_-)^{-1}$. The explicit value of $\zeta$ is computed in \cref{prop:zeta_value}, which belongs to $\bmu_{6N}$ by \cref{cor:mu_6N}.
\end{proof}

%The construction of the local quantum trace (\cite[Section 6]{Ish}) of a mapping class is not affected by this correction, since the $\bmu_N$-ambiguity is neglected by taking the absolute value. 

%\subsection{Strict extended global coefficients}

\begin{dfn}
We say that an extended global coefficient $\hatt\bbp_\Sigma$ is \emph{strict} if the phase factor $\omega$ in the pentagon relation along any pentagon in the groupoid $\mathrm{Pt}_\Sigma^{\tdot}$ is trivial. 
\end{dfn}
Recall that we have a map $\pi_N: \cF_\Sigma \to \overline{\X}_\Sigma$ into the partially compactified cluster Poisson variety, given in coordinates by $(\bp_\alpha^\tri)_{\alpha \in e(\tri)} \mapsto (Y_\alpha^\tri)_\alpha$, $Y_\alpha^\tri:=\pi(\bp_\alpha^\tri)^N$ \cite[Corollary 4.9]{Ish}. Since $\cF_N$ fibers over $B_N$, the image of $\pi_N$ is contained in
\begin{align*}
    \overline{\X}_\Sigma^{\mathrm{reg}}:= \{ \rho \in \overline{\X}_\Sigma \mid \text{$Y_\alpha^\tri(\rho) \neq -1$ for all $\tri \in \mathrm{Pt}_\Sigma$, $\alpha \in e(\tri)$}\}.
\end{align*}
Composing with the covering $\widehat{\cF}_\Sigma \to \cF_\Sigma$ (and its extension to the partial compactifications), we obtain $\widehat{\pi}_N: \widehat{\cF}_\Sigma \to \overline{\X}_\Sigma^{\mathrm{reg}}$. Recall the positive real part $\X_\Sigma(\pos) \subset \X_\Sigma$, which consists of points $\rho$ such that $Y_\alpha^\tri(\rho)>0$ for any $\tri \in \mathrm{Pt}_\Sigma$ and $\alpha \in e(\tri)$. 

\begin{lem}\label{lem:connectivity}
The subspace $\overline{\X}_\Sigma^{\mathrm{reg}}$ contains the positive real part $\X_\Sigma(\pos)$, and is path-connected in the analytic topology.
\end{lem}

\begin{proof}
The inclusion $\X_\Sigma(\pos) \subset \overline{\X}_\Sigma^{\mathrm{reg}}$ is obvious. 
Since any cluster chart provides a homeomorphism $\X_\Sigma(\pos) \cong \bR_{>0}^{|e(\tri)|}$, the positive real part is contractible. 

The ambient space $\overline{\X}_\Sigma$ is the union of complex tori $T^\tri \cong (\bC^\ast)^{|e_{\interior}(\tri)|} \times \bC^{|e(\tri) \setminus e_{\interior}(\tri)|}$ associated with $\tri \in \mathrm{Pt}_\Sigma$.\footnote{More precisely, we need the tori associated with tagged triangulations \cite{FST}. The subsequent argument is similar for these tori.}  
Let $\rho \in \overline{\X}_\Sigma^{\mathrm{reg}} \cap T^\tri$ for some $\tri \in \mathrm{Pt}_\Sigma$. 
Observe that the condition $Y_{\alpha'}^{\tri'}=-1$ for $\tri' \in \mathrm{Pt}_\Sigma$ and $\alpha' \in e(\tri')$ defines a proper analytic subset in $T^\tri$. Therefore the complement $D:=(\overline{\X}_\Sigma \setminus \overline{\X}_\Sigma^{\mathrm{reg}}) \cap T^\tri$ is a countable union of proper analytic subsets. 

It is well-known that the complement of a proper analytic subset of a connected complex manifold is path-connected (see, for instance, \cite[Theorem 2.7]{Adamus}). The same argument applies to the complement of a countable union of proper analytic subsets. 
Therefore, every point of $\overline{\X}_\Sigma^{\mathrm{reg}} \cap T^\tri$ can be connected by paths avoiding $D$. In particular, $\rho$ can be connected to a point in $\X_\Sigma(\pos)$. Since this holds for every $\tri \in \mathrm{Pt}_\Sigma$ and $\X_\Sigma(\pos)$ is path-connected, the assertion is proved.
\end{proof}

Consider the subset $\widehat{\cF}_\Sigma^{\mathrm{st}} \subset \widehat{\cF}_\Sigma$ of strict extended global coefficients. 

\begin{thm}\label{thm:strict_coeff_exist}
For any marked surface $\Sigma$, the map
\begin{align*}
    \widehat{\pi}_N: \widehat{\cF}_\Sigma^{\mathrm{st}} \to \overline{\X}_\Sigma^{\mathrm{reg}}
\end{align*}
is surjective. 
\end{thm}

\begin{proof}
Define a continuous map $s:[0,\infty) \to \overline{\cF}_N$ by
\begin{align*}
    s(Y):=\bigg( \left(\frac{Y}{1+Y}\right)^{\!1/N}, \left(\frac{1}{1+Y}\right)^{\!1/N} \bigg),
\end{align*}
where we choose the positive $N$-th roots. Then $Y \mapsto A(s(Y))$ is a continuous curve in $\bC^\ast$ with $A(s(0))=1$. Take the branch of logarithm $L(Y)=\log A(s(Y))$ along this curve with $L(0)=0$, and define
\begin{align}\label{eq:c_positive}
    c(Y):=\exp\left(-\frac{L(Y)}{N} \right).
\end{align}
Then we have $c(Y)^N=A(s(Y))^{-1}$, and hence $(s(Y);c(Y)) \in \widehat{\overline{\cF}}_N$ for $Y \geq 0$. 

\paragraph{\textbf{Lifts of positive real points.}}
Take a positive real point $\rho \in \X_\Sigma(\pos)$.
For each $\tri \in \mathrm{Pt}_\Sigma$ and $\alpha \in e(\tri)$, let $Y_\alpha^\tri=Y_\alpha^\tri(\rho) >0$ be the associated cross ratio coordinate. Then the above construction provides $\widehat{\bp}_\alpha^\tri:=(s(Y_\alpha^\tri),c(Y_\alpha^\tri)) \in \widehat{\cF}_N$. These data form an extended global coefficient, as follows. 
Let $y_\alpha^\tri:=(Y_\alpha^\tri)^{1/N}$ be the positive $N$-th root.

Recall \cite[Lemma 4.8 and Corollary 4.9]{Ish}. When considering a single flip, we will suppress triangulation labels from the notation. The parameters $Y_\alpha^\tri=Y_\alpha^\tri(\rho)$ satisfy the cluster Poisson transformation
\begin{align}\label{eq:mutation_Y}
    Y'_\alpha = \begin{cases}
        Y_\kappa^{-1} & \mbox{if $\alpha=\kappa$}, \\
        Y_\alpha(1+Y_\kappa^{-\sgn(\ve_{\alpha\kappa})})^{-\ve_{\alpha\kappa}} & \mbox{if $\alpha \neq \kappa$}.
    \end{cases}
\end{align}
Then their positive $N$-th roots satisfy
\begin{align}\label{eq:mutation_y}
    y'_\alpha = \begin{cases}
        y_\kappa^{-1} & \mbox{if $\alpha=\kappa$}, \\
        y_\alpha(1+y_\kappa^{-N\sgn(\ve_{\alpha\kappa})})^{-\ve_{\alpha\kappa}/N} & \mbox{if $\alpha \neq \kappa$}.
    \end{cases}
\end{align}
Since $s(Y_\alpha^\tri)^+/s(Y_\alpha^\tri)^- = y_\alpha^\tri$, this implies the mutation compatibility relation for $s(Y_\alpha^\tri)$.
The inversion condition \eqref{eq:inversion_condition} holds by \cref{prop:positive_normalization}. Hence we obtain an extended global coefficient $\{\hatt\bp^\tri\}_{\tri\in \mathrm{Pt}_\Sigma}$.

With this choice, the evaluation argument in \cite[Section 9.1.1]{Ish} can be recovered as follows.
We have
\begin{align}\label{eq:CQL_op_at0}
    \Psi_{(s(0);c(0))}(\opX) =\Psi_{(0,1);1}(\opX)=\mathsf{1}
\end{align}
for any cyclic operator $\opX \in \mathrm{End}(V_N)$. %Indeed, this follows from $c(0)=1$ and $\lim_{Y \to +0}w(s(Y)|k)=1$ for all $k \in \bZ_N$.

Consider a pentagon in $\mathrm{Pt}_\Sigma^{\tdot}$ as shown in \cref{fig:pentagon_proof}, where $(Y,Z,Y',Z',Y'') \in \bR_{\geq 0}^5$ are the cross ratio parameters of $\rho$ assigned to the corresponding edges, which are related as
\begin{align*}
    Y'=Y(1+Z^{-1})^{-1}, \quad Z'=Z(1+Y), \quad Y''=Y(1+Z')^{-1}
\end{align*}
by the cluster Poisson transformation. 
Let 
\begin{align*}
    (\hatt\bp,\hatt\br,\hatt\bp',\hatt\br',\hatt\bp''):=((s(Y);c(Y)),(s(Z);c(Z)),(s(Y');c(Y')),(s(Z');c(Z')),(s(Y'');c(Y''))),
\end{align*}
which therefore satisfies the mutation-compatible relation in \eqref{eq:pentagon_coeff_rel}.
The phase factor $\omega=\omega(\hatt\bp,\hatt\br,\hatt\bp',\hatt\br',\hatt\bp'') \in \bmu_N$ is a continuous function of $Y$ and $Z$, and hence a constant. 
%The limit $Z \to +0$
The evaluation at $Z=0$ with $Y>0$ produces
\begin{align*}
    (Y',Z',Y'') = (0,0,Y).
\end{align*}
%The evaluation at this point shows $\omega=1$. 
By evaluating the pentagon relation \eqref{eq:pentagon} at this point, we obtain $\Psi_{\hatt\bp}(\opU)=\omega \Psi_{\hatt\bp''}(\opU)$ by \eqref{eq:CQL_op_at0}, and hence $\omega=1$.
Therefore, $\omega=1$ holds for any $Y,Z>0$. This shows that
any positive real point $\rho \in \X_\Sigma(\pos)$ lifts to a strict extended global coefficient.

\begin{figure}[ht]
    \centering
\begin{tikzpicture}[scale=1.2]
\foreach \i in {0,1,2,3,4} {
\draw(\i*72+90:3) coordinate(A\i);
}
\draw[->,shorten >=1.44cm, shorten <=1.44cm] (A0) -- (A1);
\draw[->,shorten >=1.44cm, shorten <=1.44cm] (A1) -- (A2);
\draw[->,shorten >=1.44cm, shorten <=1.44cm] (A0) -- (A4);
\draw[->,shorten >=1.44cm, shorten <=1.44cm] (A4) -- (A3);
\draw[->,shorten >=1.44cm, shorten <=1.44cm] (A3) -- (A2);
\draw(A0) node[fill=white,circle,inner sep=0.1pt]{
\tikz[scale=0.36]{
\foreach \i in {0,1,2,3,4} {
\fill(\i*72+90:3) circle(1.3pt) coordinate(B\i);
\draw[blue] (\i*72+90:3) -- (\i*72+90+72:3);
\draw(72+90:2.5) node[scale=0.5]{$\ast$};
\draw(144+90:2.5)++(0.3,0.2) node[scale=0.5]{$\ast$};
\draw(216+90:2.5)++(0.25,0.5) node[scale=0.5]{$\ast$};
}
\draw[blue] (B0)--(B2);
\draw[blue] (B0)--(B3);
% \node[scale=0.8] at (-2,0){$u$};
% \node[scale=0.8] at (0,-0.5){$v$};
% \node[scale=0.8] at (2,0){$w$};
\node[blue,scale=0.8] at (1.1,1){$Y$};
\node[blue,scale=0.8] at (-1.1,1){$Z$};
}};
\draw(A1) node[fill=white,circle,inner sep=0.1pt]{
\tikz[scale=0.36]{
\foreach \i in {0,1,2,3,4} {
\fill(\i*72+90:3) circle(1.3pt) coordinate(B\i);
\draw[blue] (\i*72+90:3) -- (\i*72+90+72:3);
\draw(72+90:2.5) node[scale=0.5]{$\ast$};
\draw(144+90:2.5)++(0.3,0.3) node[scale=0.5]{$\ast$};
\draw(216+90:2.5) node[scale=0.5]{$\ast$};
}
\draw[blue] (B0)--(B2);
\draw[blue] (B2)--(B4);
% \node[scale=0.8] at (-2,0){$u$};
% \node[scale=0.8] at (0,0){$v$};
% \node[scale=0.8] at (0.5,-1.8){$w$};
\node[blue,scale=0.8] at (1.5,0.5){$Y^{-1}$};
\node[blue,scale=0.8] at (-1.1,1){$Z'$};
}};
\draw(A2) node[fill=white,circle,inner sep=0.1pt]{
\tikz[scale=0.36]{
\foreach \i in {0,1,2,3,4} {
\fill(\i*72+90:3) circle(1.3pt) coordinate(B\i);
\draw[blue] (\i*72+90:3) -- (\i*72+90+72:3);
\draw(72+90:2.5)++(0.5,0.35) node[scale=0.5]{$\ast$};
\draw(144+90:2.5)++(-0.2,0.1) node[scale=0.5]{$\ast$};
\draw(216+90:2.5) node[scale=0.5]{$\ast$};
}
\draw[blue] (B1)--(B4);
\draw[blue] (B2)--(B4);
% \node[scale=0.8] at (0,1.5){$u$};
% \node[scale=0.8] at (-1,-0.5){$v$};
% \node[scale=0.8] at (0.5,-1.8){$w$};
\node[blue,scale=0.8] at (1,0){${Y''}^{-1}$};
\node[blue,scale=0.8] at (-1.1,0.4){${Z'}^{-1}$};
}};
\draw(A3) node[fill=white,circle,inner sep=0.1pt]{
\tikz[scale=0.36]{
\foreach \i in {0,1,2,3,4} {
\fill(\i*72+90:3) circle(1.3pt) coordinate(B\i);
\draw[blue] (\i*72+90:3) -- (\i*72+90+72:3);
\draw(72+90:2.5)++(0.5,0.4) node[scale=0.5]{$\ast$};
\draw(144+90:2.5) node[scale=0.5]{$\ast$};
\draw(216+90:2.5)++(0.2,0.3) node[scale=0.5]{$\ast$};
}
\draw[blue] (B1)--(B3);
\draw[blue] (B1)--(B4);
% \node[scale=0.8] at (0,1.5){$u$};
% \node[scale=0.8] at (-1.5,-1){$v$};
% \node[scale=0.8] at (-0.5,0){$w$};
\node[blue,scale=0.8] at (1,-1){$Y''$};
\node[blue,scale=0.8] at (1.5,0.4){${Y'}^{-1}$};
}};
\draw(A4) node[fill=white,circle,inner sep=0.1pt]{
\tikz[scale=0.36]{
\foreach \i in {0,1,2,3,4} {
\fill(\i*72+90:3) circle(1.3pt) coordinate(B\i);
\draw[blue] (\i*72+90:3) -- (\i*72+90+72:3);
\draw(72+90:2.5)++(0.2,0.1) node[scale=0.5]{$\ast$};
\draw(144+90:2.5) node[scale=0.5]{$\ast$};
\draw(216+90:2.5)++(0.25,0.5) node[scale=0.5]{$\ast$};
}
\draw[blue] (B1)--(B3);
\draw[blue] (B0)--(B3);
% \node[scale=0.8] at (-0.5,1){$u$};
% \node[scale=0.8] at (-1.5,-1){$v$};
% \node[scale=0.8] at (2,0){$w$};
\node[blue,scale=0.8] at (1.1,1){$Y'$};
\node[blue,scale=0.8] at (-0.5,0){$Z^{-1}$};
}};
\end{tikzpicture}
    \caption{A pentagon in $\mathrm{Pt}_\Sigma^{\tdot}$}
    \label{fig:pentagon_proof}
\end{figure}

\paragraph{\textbf{Analytic continuation.}} By \cref{lem:connectivity}, $\overline{\X}_\Sigma^{\mathrm{reg}}$ is path-connected in the analytic topology, and moreover contains $\X_\Sigma(\pos)$. For any point $\rho \in \overline{\X}_\Sigma^{\mathrm{reg}}$, choose a continuous path $\rho_t$ in $\overline{\X}_\Sigma^{\mathrm{reg}}$ such that $\rho_0 \in \X_\Sigma(\pos)$, $\rho_t \in \X_\Sigma^{\mathrm{reg}}$ for $0 \leq t < 1$, and $\rho_1=\rho$. Let $Y_\alpha^\tri(t):=Y_\alpha^\tri(\rho_t) \in \bC \setminus \{-1\}$. 
Then we get the analytic continuations
\begin{align*}
    %\hatt\bp_\alpha^\tri(t) :=(s(Y_\alpha^\tri(\rho_t)),c(Y_\alpha^\tri(\rho_t))), \quad 
    &y_\alpha^\tri(t)=Y_\alpha^\tri(t)^{1/N}, \\
    &\bp_\alpha^\tri(t)=\bigg(\bigg(\frac{Y_\alpha^\tri(t)}{1+Y_\alpha^\tri(t)} \bigg)^{\!1/N}, \bigg(\frac{1}{1+Y_\alpha^\tri(t)} \bigg)^{\!1/N} \bigg), \\
    &c_\alpha^\tri(t)=A(\bp_\alpha^\tri(t))^{-1/N}
\end{align*}
along this path from the initial values at $t=0$. The parameters $Y_\alpha^\tri(t)$ satisfy the cluster Poisson transformation \eqref{eq:mutation_Y}. By taking the $N$-th roots, we see that the ratio
\begin{align*}
    % \frac{y'_\alpha(t)}{y_\alpha(t)(1+y_\kappa(t)^{-N\sgn(\ve_{\alpha\kappa})})^{-\ve_{\alpha\kappa}/N}}
    \frac{y'_\alpha(t)}{y_\alpha(t)(p_\kappa^\pm(t))^{\ve_{\alpha\kappa}}}
\end{align*}
for $\alpha\neq \kappa$ takes values in $\bmu_N$, and is continuous. Since it equals $1$ at $t=0$, it is constantly $1$. The case $\alpha=\kappa$ is treated similarly. 
Therefore the mutation-compatibility holds at $t=1$ as well. Note that this argument is also valid for the boundary case $\rho \in \overline{\X}_\Sigma^{\mathrm{reg}} \setminus \X_\Sigma^{\mathrm{reg}}$ by continuity. 
In order to see the inversion condition \eqref{eq:inversion_condition}, we first note 
\begin{align}\label{eq:inversion_continued}
    c_\alpha^\tri(t)c_\alpha^{\tri'}(t)\zeta_\inv^{-1} \in \bmu_N.
\end{align}
Indeed, by the mutation-compatibility of the underlying global coefficient and the proof of \cref{lem:involution}, we have
\begin{align*}
    (c_\alpha^\tri(t)c_\alpha^{\tri'}(t))^N= A(\bp_\alpha^\tri(t))^{-1}A(\bp_\alpha^{\tri'}(t))^{-1} = \zeta_\inv^N.
\end{align*}
Since the left-hand side of \eqref{eq:inversion_continued} is continuous for $t$ and equals $1$ at $t=0$, we have $c_\alpha^\tri(t)c_\alpha^{\tri'}(t)=\zeta_\inv$ for all $t$, especially at $t=1$. 

Thus the parameters $\hatt\bp_\alpha^\tri(t)$ form an extended global coefficient. The phase factor $\omega=\omega(t) \in \bmu_N$ also varies continuously, and equals to $1$ at $t=0$. 
Hence we obtain a strict extended global coefficient as a lift of $\rho$.
\end{proof}

% \begin{rem}
% The limit argument $Y \to +0$ can be made algebraic by considering a partial compactification $\overline{\cF}_N \to B_N \cup \{0\}$ as in \cite{Ish}, and the corresponding extended version.
% \end{rem}

\begin{cor}
For any $\hatt\bbp_\Sigma \in \widehat{\cF}_\Sigma^{\mathrm{st}}$, the pentagon relation in \cite[Theorem 1]{Ish} and the operator identities in \cite[Theorem 2]{Ish} hold true. 
\end{cor}

%\appendix

\section{Phase factors \texorpdfstring{$A_{L,N}$ and $B_{L,N}$}{A(L,N) and B(L,N)}}\label{sec:other_root}
Let $N \geq 3$ be an odd integer, and $0 \leq a \leq N-1$ be such that $\gcd(a,N)=1$. Consider the primitive $N$-th root of unity 
\begin{align*}
    q^2=e^{2\pi ia/N}.%, \quad q:=(q^2)^{(N+1)/2}
\end{align*}
Define
\begin{align*}
    L:=\begin{cases}
        N-a & \mbox{if $a$ is odd}, \\
        -a & \mbox{if $a$ is even}.
    \end{cases}
\end{align*}
Then $L$ is even, and 
\begin{align*}
    q:=(q^2)^{(N+1)/2} = e^{-\pi i L/N}.
\end{align*}
Our $L/N$ corresponds to Kashaev's $M/N$ \cite{Kas00}. In this section, we investigate the constants $A_{L,N}$ and $B_{L,N}$ introduced in \cite[p.580]{Kas00}. Following \cite{Kas00}, let $x=y[x]_y + \{x\}_y$ for $x \in \bR$ and $y \in \bZ_{>0}$, where $[x]_y$ and $\{x\}_y$ are the quotient and the remainder upon division of $x$ by $y$. For instance, $\{x\}_1$ is the fractional part of $x$, and $\{[x]_1\}_2$ is the parity of the integer part of $x$. Then, Kashaev's constants are
\begin{align*}
    A_{L,N} &:= \frac{1}{N} \sum_{k=1}^{N-1} k \bigg( \bigg\{\frac{(2k+1)L}{N}\bigg\}_1 - \bigg\{\bigg[\frac{(2k+1)L}{N}\bigg]_1\bigg\}_2 \bigg), \\
    B_{L,N}&:= \frac{1}{N} \sum_{k=1}^{N-1} k \bigg( \bigg\{\frac{2kL}{N}\bigg\}_1 - \bigg\{\bigg[\frac{2kL}{N}\bigg]_1\bigg\}_2 -1\bigg).
\end{align*}
Let us first rewrite them as follows. 
Write $N=2n+1$ and $L=2\ell$. For an integer $x$, 
let 
\begin{itemize}
    \item $\{x\}_N \in \{0,\dots,N-1\}$ denote the residue of $x$ mod $N$.
    \item $\langle x\rangle_N \in \{-n,\dots,n\}$ denote the residue of $x$ mod $N$. 
\end{itemize}
These two residues are related by
\begin{align}\label{eq:residue_relation}
    \langle x \rangle_N = \{2x\}_N - \{x\}_N.
\end{align}
Let $c_k:=\langle (2k+1)\ell \rangle_N$, $d_k:=\{2k\ell\}_N$ for $k=1,\dots,N-1$. 
Using the elementary properties
\begin{align*}
    \left\{\frac{2x}{N} \right\}_1 - \left\{\left[\frac{2x}{N} \right]_1 \right\}_2 = \frac{2\langle x \rangle_N}{N}, \qquad \left\{\frac{2x}{N} \right\}_1 + \left\{\left[\frac{2x}{N} \right]_1 \right\}_2 = \frac{2\{ x \}_N}{N}
\end{align*}
for any $x \in \bZ$, we can rewrite Kashaev's constants as
\begin{align*}
    A_{L,N} &= \frac{1}{N}\sum_{k=1}^{N-1} k \frac{2\langle (2k+1)\ell\rangle_N}{N} = \frac{2}{N^2}\sum_{k=1}^{N-1} kc_k, \\
    B_{L,N} &= \frac{1}{N}\sum_{k=1}^{N-1} k \bigg(\frac{2\{ 2k\ell\}_N}{N}-1 \bigg) = \frac{1}{N}\sum_{k=1}^{N-1} k \bigg(\frac{2d_k}{N}-1 \bigg).
\end{align*}

\begin{lem}\label{prop:A_gamma}
We have the relation
\begin{align*}
    \prod_{k \in \bZ_N} \gamma(k)= e^{\pi i N A_{L,N}}.
    %\zeta_\inv^N = e^{-\pi i N A_{L,N}}, \quad %\zeta=e^{\frac{\pi i}{3}(B_{L,N}-A_{L,N})}
\end{align*}
%under a suitable choice of branches of roots.
\end{lem}

\begin{proof}
% The defining property of $\zeta_\inv$ is  (recall the proof of \cref{lem:involution}):
% \begin{align*}
%     \zeta_\inv^N=\prod_{k \in \bZ_N} \gamma(k)^{-1}.
% \end{align*}
% We can write $q=e^{-\pi i L/N}$, and hence
We compute
\begin{align*}
    %\zeta_\inv^N=
    \prod_{k \in \bZ_N} \gamma(k) = \exp\bigg(-\pi i \frac{L}{N} \sum_{k=1}^{N-1}k^2\bigg) = \exp\bigg(-\pi i \frac{L}{N} \frac{(N-1)N(2N-1)}{6}\bigg).
\end{align*}
On the other hand,
\begin{align*}
    N A_{L,N} = \frac{L}{N}\sum_{k=1}^{N-1}k(2k+1) =\frac{L}{N}\frac{(N-1)N(4N+1)}{6} \equiv -\frac{L}{N}\frac{(N-1)N(2N-1)}{6} \quad (\mbox{mod $2$}).
\end{align*}
Therefore the asserted relation holds. 
%Therefore, $\zeta_\inv^N=e^{-\pi i N A_{L,N}}$. Using the freedom observed in \cref{rem:zeta_inv_choice}, we can \emph{choose} the $N$-th root $\zeta_\inv:=e^{-\pi i A_{L,N}}$. This proves the first relation.

\end{proof}
Recall the function \eqref{eq:c_positive}. 
The following recovers \cite[Theorem 2.9 (2)]{Ish} at positive points:
\begin{prop}\label{prop:positive_normalization}
Let $\zeta_\inv:=e^{-\pi i A_{L,N}}$. 
Then
\begin{align*}
    c(Y) c(Y^{-1}) = \zeta_\inv
\end{align*}
for any $Y>0$. 
\end{prop}

\begin{proof}
The left-hand side is continuous for $Y>0$. 
First note that $s(Y^{-1})=s(Y)^{-1}$ from the construction. Then we have (cf. the proof of \cref{lem:involution})
\begin{align*}
    (c(Y)c(Y^{-1}))^N = A(s(Y^{-1}))^{-1}A(s(Y))^{-1} = \zeta_\inv^N.
\end{align*}
Therefore $c(Y)c(Y^{-1})\zeta_\inv^{-1} \in \bmu_N$, which is a constant for $Y>0$ by continuity. 

It is immediate from the definition that $c(0)=1$. It suffices to show 
\begin{align}\label{eq:c_limit_infty}
    \lim_{Y \to +\infty} c(Y)=\zeta_\inv,
\end{align}
which concludes $c(Y) c(Y^{-1}) = \zeta_\inv$ for any $Y>0$. From the expression
\begin{align*}
    A(y)=\prod_{j=1}^N (1+q^{2j-1}y)^{M-j},
\end{align*}
we have
\begin{align*}
    L(Y)=\sum_{j=1}^N (M-j) \Log (1+ q^{2j-1}Y^{1/N})
\end{align*}
for $Y>0$, 
where $Y^{1/N}>0$ is the positive $N$-th root and $\Log$ denotes the principal branch of logarithm in $(-\pi,\pi)$. By $q^2=e^{2\pi i a/N}$, we have $q^{2j-1}=\exp(\frac{2\pi i a(j-M)}{N})$. Hence,
\begin{align*}
    \lim_{y \to +\infty}(\Log (1+ q^{2j-1}y) - \log y) = \frac{2\pi i}{N}\langle a(j-M) \rangle_N.
\end{align*}
Since $\sum_{j=1}^N (M-j)=0$, we obtain
\begin{align*}
    \lim_{y \to +\infty} L(Y) &= \lim_{y \to +\infty}\sum_{j=1}^N (M-j)(\Log (1+ q^{2j-1}y) - \log y + \log y) \\
    &= \sum_{j=1}^N (M-j)\frac{2\pi i}{N}\langle a(j-M) \rangle_N \\
    &= -\frac{2\pi i}{N}\sum_{k=-n}^n k\langle ak \rangle_N.
\end{align*}
We note
\begin{align}\label{eq:c_k_formula}
    c_k = \langle (2k+1)\ell \rangle_N = \langle 2\ell(k+n+1) \rangle_N = -\langle a(k+n+1) \rangle_N =  -\langle a(k-n) \rangle_N,
\end{align}
where we used $2(n+1)\equiv 1$ mod $N$, $2\ell=L \equiv -a$ mod $N$, and $N=2n+1$. Hence
\begin{align*}
    A_{L,N} &= -\frac{2}{N^2}\sum_{k=0}^{N-1} k\langle a(k-n) \rangle_N \\
    &= -\frac{2}{N^2}\sum_{k=-n}^{n} (k+n)\langle ak \rangle_N.
\end{align*}
Since $\gcd(a,N)=1$, $k \mapsto \langle ak \rangle_N$ gives a bijection on $\{-n,\dots,n\}$.
%the centered residues of $ak$ for $k\in \{-n,\dots,n\}$ exhaust $\{-n,\dots,n\}$.
Therefore $\sum_{k=-n}^n\langle ak \rangle_N=0$, and hence the second term vanishes. Combining the arguments, we obtain
\begin{align*}
    \lim_{Y \to +\infty} c(Y) &= \lim_{Y \to +\infty} \exp\bigg(-\frac{L(Y)}{N}\bigg) \\
    &= \exp\bigg( \frac{2\pi i}{N^2}\sum_{k=-n}^n k\langle ak \rangle_N\bigg) \\
    &= \exp(-\pi i A_{L,N}) =\zeta_\inv,
\end{align*}
which proves \eqref{eq:c_limit_infty}. 
\end{proof}

Recall the normalization constant $\zeta_A$ \cite[Section 4.4]{Ish} that ensures $\opA_v^3=\mathsf{1}$. The required relation is $\zeta_A^3=(N\zeta_-)^{-1}$, and we have a cube-root ambiguity for the choice of $\zeta_A$. Here,
\begin{align*}
    \zeta_-:=\sum_{k \in \bZ_N}\gamma(k)^{-1}
\end{align*}
is the constant appearing in the Fourier transformation formula \cite[Lemma 3.12]{Ish}. 

\begin{prop}\label{prop:zeta_value}
For an appropriate choice of $\zeta_A$, we have
\begin{align*}
    \zeta := \zeta_\inv (\zeta_A\zeta_-)^{-1} = e^{\frac{\pi i}{3}(B_{L,N}-A_{L,N})}.
\end{align*}
\end{prop}

\begin{proof}
Writing $L=2\ell$, we obtain
\begin{align*}
    \zeta_-= \sum_{k \in \bZ_N} \exp\bigg(\frac{2\pi i\ell k^2}{N} \bigg) = \ve_N \bigg(\frac{\ell}{N}\bigg)_{\!\!J}\sqrt{N}
\end{align*}
by the usual quadratic Gauss summation\footnote{This approach avoids the use of the more general Gauss summation \cite[Theorem 3.11]{Ish}.},
where
\begin{align*}
    \ve_N:=\begin{cases}
        1 & \mbox{if $N\equiv 1$ mod $4$}, \\
        i & \mbox{if $N\equiv 3$ mod $4$},
    \end{cases}
\end{align*}
and $\big(\frac{\ell}{N}\big)_{\!J} \in \{\pm 1\}$ denotes the Jacobi symbol. In particular, we have
\begin{align*}
    \frac{\zeta_-^2}{N}= \ve_N^2 = (-1)^{(N-1)/2}.
\end{align*}
On the other hand, by \cref{lem:2A+B} below, we have
\begin{align*}
    e^{-\pi i (2A_{L,N}+B_{L,N})} = e^{\pi i \frac{N-1}{2}} = (-1)^{(N-1)/2}.
\end{align*}
Therefore
\begin{align*}
    \frac{\zeta_-^2}{N} = e^{-\pi i (2A_{L,N}+B_{L,N})}.
\end{align*}
By the normalization $\zeta_A^3=(N\zeta_-)^{-1}$, we have $(\zeta_A\zeta_-)^3 = \zeta_-^2/N$. We thus \emph{choose} a cube root $\zeta_A$ so that
\begin{align*}
    \zeta_A \zeta_- =e^{-\frac{\pi i}{3} (2A_{L,N}+B_{L,N})}.
\end{align*}
Then we obtain %\cite[Section 9.2.4]{Ish}
\begin{align*}
    \zeta = \zeta_\inv (\zeta_A\zeta_-)^{-1} = e^{\frac{\pi i}{3}(B_{L,N}-A_{L,N})}.
\end{align*}

\end{proof}

\begin{lem}\label{lem:2A+B}
We have $2A_{L,N}+ B_{L,N} \equiv \frac{N-1}{2}$ mod $2$.
%Write $N=2n+1$ and $L=2\ell$. For an integer $x$, let $\langle x\rangle_N \in \{-n,\dots,n\}$ be the ``centered'' residue of $x$ mod $N$. Then we have
% \begin{enumerate}
%     \item $A_{L,N}=\displaystyle\frac{2}{N^2}\sum_{k=1}^{N-1} kc_k$ and $B_{L,N}=\displaystyle\frac{1}{N}\sum_{k=1}^{N-1} k \bigg(\frac{2d_k}{N}-1 \bigg)$.
%     \item %$2A_{L,N}+ B_{L,N} = 2\nu_{L,N} - \frac{N-1}{2}$. 
%     %In particular, 
%     $2A_{L,N}+ B_{L,N} \equiv \frac{N-1}{2}$ mod $2$.
% \end{enumerate}
%In particular, each term in the right-hand side is $\pm 1$, and hence $2A_{L,N}+ B_{L,N} \equiv r$ mod $2$.
\end{lem}

\begin{proof}
%We first note $c_k = \langle 2\ell(k+n+1)\rangle_N$, since $2(n+1)\equiv 1$ mod $N$.
For any integer $a$ satisfying $\gcd(a,N)=1$, the map $j \mapsto \{aj\}_N$ gives a bijection on $\bZ_N$. We will use this type of ``change of coordinates'' below.
For example, 
\begin{align}\label{eq:V-total}
    \sum_{j=0}^{N-1} \{aj\}_N = \sum_{j=0}^{N-1} j =Nn.
\end{align}
% We claim the summation formula
% \begin{align}
%     2\sum_{k=1}^{N-1} kc_k + \sum_{k=1}^{N-1} kd_k = N^2 \nu_{L,N},
% \end{align}
% where $\nu_{L,N}:=\#\{ 0 \leq k \leq N-1 \mid c_k <0\}$. 
Consider
\begin{align*}
    U(a):=\sum_{j=1}^{N-1} j \{aj\}_N, \quad V(a):=\sum_{j=1}^n \{aj\}_N
\end{align*}
for any $a$ satisfying $\gcd(a,N)=1$. 
Using $U(a)$ and $V(a)$,  we can write Kashaev's constants as 
\begin{align}
    A_{L,N} &=\frac{2}{N^2}(U(4\ell)-U(2\ell)) + \frac{2}{N}(V(4\ell)-V(2\ell)), \label{eq:A-formula}\\
    B_{L,N} &= \frac{2}{N^2}U(2\ell) - n. \label{eq:B-formula}
\end{align}
Indeed, the formula for $B_{L,N}$ is straightforward. To prove the formula for $A_{L,N}$, first note $c_k=\langle 2\ell(k+n+1)\rangle_N=\{4\ell(k+n+1)\}_N - \{2\ell(k+n+1)\}_N$ by \eqref{eq:c_k_formula} and \eqref{eq:residue_relation}.
By the shift of variable $j=\{k+n+1\}_N$ in $\bZ_N$, we get
\begin{align*}
    \sum_{k=0}^{N-1}k\{a(k+n+1)\}_N &= \sum_{j=0}^n (j-n-1+N)\{aj\}_N + \sum_{j=n+1}^{N-1}(j-n-1)\{aj\}_N \\
    &= U(a) -n(n+1)N + NV(a)
\end{align*}
for any $a$ such that $\gcd(a,N)=1$. 
Hence we obtain
\begin{align*}
    A_{L,N} 
    &= \frac{2}{N^2} \sum_{k=0}^{N-1} k (\{4\ell(k+n+1)\}_N - \{2\ell(k+n+1)\}_N) \\
    &= \frac{2}{N^2}(U(4\ell)-U(2\ell)) + \frac{2}{N}(V(4\ell)-V(2\ell)).
\end{align*}
This proves \eqref{eq:A-formula}.

We claim the following relation between $U$ and $V$:
\begin{align}\label{eq:U-V-relation}
    U(a)-2U(2a) = NV(2a) - N^2n
\end{align}
for $\gcd(a,N)=1$. Indeed, by the change of variable $j \mapsto \{2j\}_N$, we obtain $U(a)=\sum_{j=1}^{N-1} \{2j\}_N\{2aj\}_N$, and hence
\begin{align*}
    U(a) - 2U(2a) &= \sum_{j=1}^{N-1}(\{2j\}_N -2j)\{2aj\}_N \\
    &= -N \sum_{j=n+1}^{N-1}\{2aj\}_N \\
    &= -N \bigg( \sum_{j=1}^{N-1}\{2aj\}_N - \sum_{j=1}^n \{2aj\}_N \bigg) = -N^2 n + NV(2a),
\end{align*}
where we used \eqref{eq:V-total}.

By combining \eqref{eq:A-formula}, \eqref{eq:B-formula} and \eqref{eq:U-V-relation}, we obtain
\begin{align*}
    2A_{L,N} + B_{L,N} &= \frac{2}{N^2}(2U(4\ell)-U(2\ell)) + \frac{4}{N}(V(4\ell)-V(2\ell)) -n \\
    &= n + \frac{2V(4\ell)-4V(2\ell)}{N} \\
    &= n + \frac{2}{N}\sum_{j=1}^n(\{4\ell j\}_N -2\{2\ell j\}_N) \\
    &\equiv n \mod 2.
\end{align*}
Here, in the last line, we have used $\{4\ell j\}_N -2\{2\ell j\}_N \equiv 0$ mod $N$. This completes the proof.
\end{proof}

\begin{cor}\label{cor:mu_6N}
We have $\zeta \in \bmu_{6N}$.
\end{cor}

\begin{proof}
It suffices to prove $N(B_{L,N}-A_{L,N}) \in \bZ$. Since $2A_{L,N} +B_{L,N} \in \bZ$, it suffices to show $3N A_{L,N} \in \bZ$. Since $\langle x \rangle_N - x \in N\bZ$, we obtain
\begin{align*}
    3N A_{L,N} &= \frac{6}{N} \sum_{k=1}^{N-1} k\langle (2k+1)\ell\rangle_N \\
     &\equiv \frac{6}{N} \ell\sum_{k=1}^{N-1} k(2k+1) \mod \bZ\\
     &\equiv (N-1)(4N+1)\ell \equiv 0 \mod \bZ.
\end{align*}
Therefore $3N A_{L,N} \in \bZ$.
\end{proof}

\end{document}